\documentclass[11pt]{article}

\usepackage[all,cmtip]{xy}
\usepackage{amsmath}
\usepackage{amssymb}
\usepackage{amsthm}
\usepackage{fullpage}

\newtheorem*{example*}{Example}
\newtheorem*{remark*}{Remark}
\newtheorem*{theorem*}{Theorem}
\newtheorem*{definition*}{Definition}
\newtheorem{theorem}{Theorem}[section]
\newtheorem{rem}[theorem]{Remark}
\newenvironment{remark}{\begin{rem}\rm}{\end{rem}}

\newtheorem{proposition}[theorem]{Proposition}

\newtheorem{lemma}[theorem]{Lemma}
\newtheorem{corollary}[theorem]{Corollary}
\newtheorem{eg}[theorem]{Example}
\newenvironment{example}{\begin{eg}\rm}{\end{eg}}
\newtheorem{definition}[theorem]{Definition}
\newtheorem{conj}[theorem]{Conjecture}

\newtheorem{prob}[theorem]{Problem}

\newcommand{\R}{\mathbb{R}}
\newcommand{\C}{\mathbb{C}}
\newcommand{\Z}{\mathbb{Z}}
\newcommand{\W}{\mathcal{W}}

\newcommand{\balpha}{\boldsymbol{\alpha}}

\DeclareMathOperator{\Hom}{Hom}
\DeclareMathOperator{\supp}{Supp}
\DeclareMathOperator{\Lie}{Lie}
\DeclareMathOperator{\End}{End}
\DeclareMathOperator{\Diff}{Diff}
\DeclareMathOperator{\Ch}{Ch}
\DeclareMathOperator{\Str}{Str}
\DeclareMathOperator{\Id}{Id}
\DeclareMathOperator{\im}{im}
\DeclareMathOperator{\Td}{Td}
\DeclareMathOperator{\Tr}{Tr}
\DeclareMathOperator{\Th}{Th}

\DeclareMathOperator{\ind}{index}

\DeclareMathOperator{\rank}{rank}
\DeclareMathOperator{\vol}{vol}

\DeclareMathOperator{\hol}{hol-ind}
\DeclareMathOperator{\inl2}{ind}
\DeclareMathOperator{\ad}{ad}

\begin{document}

\title{An equivariant index formula for almost-CR manifolds}
\author{Sean Fitzpatrick\\
University of Toronto}
\maketitle
\abstract{We consider a consider the case of a compact manifold $M$, together with the following data: the action of a compact Lie group $H$ and a smooth $H$-invariant distribution $E$, such that the $H$-orbits are transverse to $E$.  These data determine a natural equivariant differential form with generalized coefficients $\mathcal{J}(E,X)$ whose properties we describe.

When $E$ is equipped with a complex structure, we define a class of symbol mappings $\sigma$ in terms of the  resulting almost-CR structure that are $H$-transversally elliptic whenever the action of $H$ is transverse to $E$.  We determine a formula for the $H$-equivariant index of such symbols that involves only $\mathcal{J}(E,X)$ and standard equivariant characteristic classes. This formula generalizes the formula given in \cite{F} for the case of a contact manifold.}
\section{Introduction}
Let $M$ be compact manifold equipped with an action $\Phi:H\rightarrow 
\Diff(M)$ of a compact Lie group $H$, and suppose we are given a smooth, $H$-invariant distribution $E\subset TM$ whose anihilator $E^0\subset T^*M$ satisfies the following conditions:
\begin{itemize}
\item[(i)] $E^0$ is oriented

\item[(ii)] $E^0\cap T^*_HM = 0$,
\end{itemize}
where $T^*_HM\subset T^*M$ denotes the space of covectors orthogonal to the $H$-orbits.
Property (i) is the statement that $E$ is co-oriented.  When the pair $(E,\Phi)$ satisfies property (ii), we say that the action of $H$ is {\em transverse} to $E$.

When a pair $(E,\Phi)$ satisfying properties (i) and (ii) exists, it is 
possible to define a natural equivariant differential form with 
generalized coefficients $\mathcal{J}(E,X)$ that depends only on the 
distribution $E$ and the action $\Phi$, as follows:

Denote by $\theta$ the canonical 1-form on $T^*M$, and let $\imath:E^0\rightarrow T^*M$ and $q:E^0\rightarrow M$ denote inclusion and projection, respectively. We denote by $D\theta(X) = d\theta -\iota(X)\theta$ the equivariant differential of $\theta$.  We then define
\begin{equation}\label{jintro}
\mathcal{J}(E,X) = (2\pi i)^k q_*\imath^*e^{iD\theta(X)}\quad \mbox{for any} \ X\in\mathfrak{h},
\end{equation}
where $k=\rank E^0$ and $q_*$ denotes integration over the fibres of $E^0$.  The assumption that the action of $H$ on $M$ is transverse to $E$ implies that this fibre integral is well-defined as an oscillatory integral in the sense of H\"ormander \cite{Hor}, and determines an equivariant differential form with generalized coefficients on $M$.

This form is an extension to distributions of higher corank of the form $\alpha\wedge\delta_0(D\alpha(X))$ defined in \cite{F} for the case of a contact distribution $E=\ker \alpha$.  If one carries out the fibre integral in (\ref{jintro}) locally, in terms of some frame $\balpha = (\alpha_1,\ldots, \alpha_k)$ for $E^0$, one obtains the expression
\begin{equation}
\mathcal{J}(E,X) = \alpha_k\wedge\cdots\wedge\alpha_1\delta_0(D\balpha(X))\label{jloc},
\end{equation}
where $\delta_0$ denotes the Dirac delta function on $\R^k$.  Using the properties of $\delta_0$, one can show directly that the expression (\ref{jloc}) is independent of the choice of frame $\balpha$, and that $D\mathcal{J}(E,X) = 0$.
Let $\pi$, $r$ and $s$ be the projections given by the following diagram:
\[
\xymatrix{T^*M\ar[r]_{r}\ar[d]_\pi &E^* \ar[ld]^{s}\\
M &}
\]
Let $\W = \W^+\oplus \W^-\rightarrow M$ be an $H$-equivariant $\Z_2$-graded vector bundle, and suppose we have a symbol
\begin{equation*}
\sigma:\pi^*\W^+\rightarrow \pi^*\W^-,
\end{equation*}
on $T^*M$ such that $\sigma = r^*\sigma_E$, for some symbol $\sigma_E:s^*\W^+\rightarrow s^*\W^-$ defined on $E^*$.  If $\sigma_E$ is elliptic, then $\supp(\sigma) = E^0$, and the assumption that the action of $H$ is transverse to $E$ implies that $\supp(\sigma)\cap T^*_HM = 0$, whence $\sigma$ satisfies Atiyah's definition of an $H$-transversally elliptic symbol \cite{AT}.

\begin{example*}\rm
If $M$ is a Cauchy-Riemann (CR) manifold, and $E$ is the real distribution underlying the CR distribution $E_{1,0}\subset TM\otimes\C$, then one such $\sigma$ is the principal symbol of $\overline{\partial}_b + \overline{\partial}_b^*$, where $\overline{\partial}_b$ is the tangential CR operator.  A related example is the symbol of the differential operator constructed in \cite{F} for contact manifolds. In either of these cases, it is necessary that the $H$-action be transverse to $E$ in order to define the equivariant index of the symbol.  An example of such an action occurs in the case of a Sasakian manifold $(M,E,g)$: the action of the group of isometries of $(M,g)$ is always transverse to $E$, since the Reeb field is Killing \cite{Blair}.
\end{example*}
\begin{example*}\rm
Given a principal $H$-bundle $\pi:M\rightarrow B$, we let $E\cong \pi^*TB$ be the horizontal distribution with respect to some choice of connection.  Then, if $\sigma_B$ is an elliptic symbol on $T^*B$, its pullback to $T^*M$ will be transversally elliptic.  This example was studied in detail by Berline and Vergne \cite{BV2}, and extended by Vergne to the case of a locally free action \cite{V1}.  The formula we give below can be thought of as a further extension of these results to a broader class of group actions.
\end{example*}

Since the symbols we consider are $H$-transversally elliptic, their $H$-equivariant index is defined as a generalized function on $H$ \cite{AT}.  Cohomological formulas for this index have been given by Berline and Vergne \cite{BV1,BV2}, and recently, by Paradan and Vergne \cite{PV3}.  Both formulas involve the integral over $T^*M$ of certain characteristic equivariant forms.  The earlier formula has the difficuly of requiring the integration of forms without compact support, while the newer formula requires a choice of cutoff function.  By integrating over the fibres of $T^*M$, we obtain a formula for the index as an integral over $M$ which is free of both choices and growth conditions.

We then specialize to the case of an $H$-invariant almost-CR structure $E\otimes \C = E_{1,0}\oplus E_{0,1}\subset TM\otimes \C$.  If we let $\W = \bigwedge E^{0,1}$, then we can construct an $H$-transversally elliptic symbol $\sigma:\pi^*\W^+\rightarrow \pi^*\W^-$ that depends only on $E^*$.  The integral over the fibres of $T^*M$ of the corresponding Chern character can be computed explicitly in terms of equivariant characteristic classes on $M$, giving:
\begin{theorem*}
Let $\mathcal{V}\rightarrow M$ be an $H$-equivariant Hermitian vector bundle on $M$,  and let $\sigma_{\mathcal{V}}$ denote the symbol $\sigma\otimes\Id_{\mathcal{V}}$ on $\pi^*\W\otimes\mathcal{V}$.
Then the $H$-equivariant index of $\sigma_{\mathcal{V}}$ is the generalized function on $H$ whose germ at $1\in H$ is given, for $X\in\mathfrak{h}$ sufficiently small, by
\begin{equation}\label{compint}
\ind^H(\sigma_{\mathcal{V}})(e^X)  =
\int_{M} (2\pi i)^{-\rank E/2}\Td(E,X)\hat{A}^2(E^0,X)\mathcal{J}(E,X)\Ch(\mathcal{V})(X), 
\end{equation}
with similar formulas near other elements $h\in H$.
\end{theorem*}
This formula is similar to the one obtained in \cite{F} for contact manifolds, with the exception of the term $\hat{A}^2(E^0,X)$, whose appearance in this formula reflects the fact that $E^0$ need not be trivial in general.

As an application, suppose we are given a complex homogeneous space $H/K$, and a unitary $K$-representation $\tau:K\rightarrow \End(V_\tau)$. The associated vector bundle $\mathcal{V}_\tau = H\times_K V_\tau$ is holomorphic, and one may define two different induced representations of $H$: the Frobenius induced representation $\inl2^H_K(\tau)$ on the space of $L^2$-sections of $\mathcal{V}_\tau$, and the holomorphic induced representation $\hol^H_K$ on the space of holomorphic sections of $\mathcal{V}_\tau$.

A formula for the character of $\inl2^H_K(\tau)$ was given by Berline and Vergne \cite{BV0} as the $H$-equivariant index of the zero symbol, while the character of $\hol^H_K(\tau)$ is the $H$-equivariant index of the elliptic Dolbeault-Dirac operator (given by the equivariant Riemann-Roch formula).

We can, by varying the rank of $E$, view both of these formulas as special cases of our formula (\ref{compint}), for $\mathcal{V} = \mathcal{V}_\tau$: When $E = 0$, we have $\sigma = 0$, and obtain the Berline-Vergne formula for the character of $\inl2^H_K(\tau)$.  When $E=TM$, $\sigma$ becomes the symbol of the Dolbeault-Dirac operator, giving the character of $\hol^H_K(\tau)$.

\section{Transverse group actions}
Let $M$ be a differentiable manifold, and $E\subset TM$ a given distribution.
Suppose that a compact Lie group $H$ acts on $M$ preserving $E$, and let $\mathfrak{h}_M\subset TM$ denote the set of vectors tangent to the orbits of $H$ in $M$.  
\begin{definition}
We say the action of $H$ on $M$ is {\em transverse} to $E$ if $E+\mathfrak{h}_M = TM$.
\end{definition}
Let $\theta\in\mathcal{A}^1(T^*M)$ denote the canonical 1-form on $T^*M$, and let $f_\theta:T^*M\rightarrow \mathfrak{h}^*$ denote the corresponding moment map.  We denote by $T^*_HM = f_\theta^{-1}(0)$ the set of covectors orthogonal to the $H$-orbits.
We may then equivalently define the action of $H$ to be transverse to $E$ if it satisfies
\begin{equation}
T^*_HM\cap E^0 = \{0\}.\label{transverse}
\end{equation}

\begin{remark}\label{lfa}
The assumption of transversality implies that $\rank E^0\leq \dim H$.  In the case that $\rank E^0 = \dim H$, the action of $H$ on $M$ is locally free. More precisely, at any $x\in M$, one has $\rank E^0 \leq \dim H - \dim H_x$, whence $\rank E^0 = \dim H$ implies that $\dim H_x = 0$.  A sub-bundle $E$ that is transverse to the $H$-orbits is then the space of horizontal vectors with respect to some choice of connection, and the anihilator $E^0=M\times\mathfrak{h}^*$ is trivial.  We are then in the same setting considered in \cite{BV2} in the case of a free action, or \cite{V1}, in the orbifold setting (see Section \ref{vergnecomp} below).  We are thus considering a broader class of group actions in this paper, since any locally free action will be transverse to a horizontal distribution, but not all actions satisfying \eqref{transverse} are locally free.
\end{remark}

Given the action of a group $H$ on a set $V$ and any $h\in H$, we let $V(h)= \{v\in V|h\cdot v = v\}$ denote the corresponding set of elements fixed by the $H$-action.  For example, if $H$ is a Lie group, then $H(h)$ denotes the centralizer of $h$ in $H$, and $\mathfrak{h}(h)$ denotes its Lie algebra, the set of points fixed by $h$ under the adjoint action.  If $H$ is a compact Lie group acting on a manifold $M$, we have the decomposition
\begin{equation*}
TM|_{M(h)} = TM(h)\oplus\mathcal{N},
\end{equation*}
where $TM(h) = \ker(h - Id)$ denotes the points in $TM$ fixed by the 
action of $h$, and $\mathcal{N} = \im(h-Id)$ denotes the normal 
bundle.  From the corresponding action on $T^*M$, we have the canonical identification $T^*(M(h)) \cong 
(T^*M)(h)$. 
With respect to the action of $H(h)$ on $T^*M(h)$, we note the following lemmas:
\begin{lemma}{\rm \cite[Lemma 19]{BV1}}\label{l3}
\begin{enumerate}
\item The canonical 1-form $\theta^h$ on $T^*M(h)$ is the pullback under inclusion of the canonical 1-form $\theta$ on $T^*M$. 
\item The corresponding moment map $f_{\theta^h}:T^*M(h)\rightarrow 
\mathfrak{h}^*(h)$ is given by the restriction of $f_\theta$ to $T^*M(h)$.
\item $T^*_{H(h)}M(h) = f^{-1}_{\theta^h}(0) = (T^*_HM)(h)$.
\end{enumerate}
\end{lemma}
\begin{lemma}\label{em}
For any $h\in H$, we have the identification
\begin{equation}
\mathfrak{h}_M(h) = \mathfrak{h}(h)_M.
\end{equation}
\end{lemma}
\begin{proof}
At any $x\in M$ we have that $\displaystyle \mathfrak{h}_M|_x \cong \mathfrak{h}/\mathfrak{h}_x$.  Choose an $H_x$-equivariant splitting $s:\mathfrak{h}/\mathfrak{h}_x\rightarrow \mathfrak{h}$ of the exact sequence
\begin{equation*}
0\rightarrow \mathfrak{h}_x\rightarrow\mathfrak{h}\rightarrow\mathfrak{h}/\mathfrak{h}_x\rightarrow 0.
\end{equation*}
By the equivariance of $s$, we thus have $s\left(\left(\mathfrak{h}/\mathfrak{h}_x\right)(h)\right)\subset \mathfrak{h}(h)$, whence $\mathfrak{h}_M(h)\subset\mathfrak{h}(h)_M$.  The opposite inclusion is clear, and thus the result follows.
\end{proof}

\begin{proposition}\label{fix}
If the action of $H$ on $M$ is transverse to $E\subset TM$, then for 
any $h\in H$, the action of $H(h)$ on $M(h)$ is transverse to $E(h)\subset 
TM(h)$. 
\end{proposition}
\begin{proof}
If $H$ acts on $M$ transverse to $E$, then we have 
\begin{equation*}
TM(h) = (E+\mathfrak{h}_M)(h) = E(h) + \mathfrak{h}_M(h) = E(h)+\mathfrak{h}(h)_M,
\end{equation*}
by averaging with respect to the subgroup generated by $h$, and then using Lemma \ref{em}.
\end{proof}

\section{Equivariant differential forms with generalized coefficients}
Let $N$ be a smooth manifold, not necessarily compact, equipped with the action of a Lie group $G$.  Let 
$\mathcal{A}^\infty(\mathfrak{g}, N)$ denote the complex of smooth 
equivariant differential forms on $N$. These are the smooth maps 
$\alpha :\mathfrak{g} = \Lie(G) \rightarrow \mathcal{A}(N)$ that are equivariant 
with respect to the actions of $G$ on $\mathfrak{g}$ and $N$.
The space $\mathcal{A}^\infty(\mathfrak{g},N)$ is equipped with the 
equivariant differential $D$, given by
\begin{equation*}
(D\alpha)(X) = d(\alpha(X)) - \iota(X_N)\alpha(X),
\end{equation*}
for any $X\in\mathfrak{g}$, where $\iota(X_N)$ is contraction by the 
fundamental vector field on $N$ generated by the infinitesimal action of 
$X\in\mathfrak{g}$.  We can pass to the complex $\mathcal{A}^{-\infty}(\mathfrak{g},N)$ 
of equivariant differential forms with generalized coefficients by allowing $G$-equivariant $C^{-\infty}$ maps from $\mathfrak{g}$ to $\mathcal{A}(N)$ \cite{KV}.  That is, $\alpha\in\mathcal{A}^{-\infty}(\mathfrak{g},N)$ if for any compactly supported test function $\phi\in C^\infty(\mathfrak{g})$, the pairing
$\displaystyle \int_{\mathfrak{g}}\alpha(X)\phi(X)dX$ defines a smooth differential form on $N$.  The equivariant differential $D$ extends to $\mathcal{A}^{-\infty}(\mathfrak{g},N)$, and in either space we have $D^2=0$, so that one may define the equivariant cohomology spaces $\mathcal{H}^{\pm\infty}(\mathfrak{g},N)$.

\begin{example}
Let $G=S^1$ act on $N=S^1$ by rotation.  Suppose 
$\alpha\in\mathcal{A}^{-\infty}(\mathfrak{g},N)$ has odd degree.  Then $\alpha(\xi) = f(\xi,\eta)d\eta$, where $f$ depends smoothly on 
the 
coordinate $\eta$ on $N$, and distributionally on the $G$ coordinate $\xi$.  
Then $D\alpha(\xi) = \xi f(\xi,\eta)$, so that $D\alpha = 0$ if and only if 
$f(\xi,\eta) = h(\eta)\delta(\xi)$ for some smooth function $h(\eta)$, where $\delta(\xi)$ denotes the Dirac measure on $S^1$.  The 
exactness of  
$\alpha$ is equivalent to $h(\eta)$ being a derivative, from which we see 
that the odd 
part of $\mathcal{H}^{-\infty}(\mathfrak{g},N)$ is generated by 
$\delta(\xi)d\eta$.
\end{example}

\begin{example}\cite{PV3}\label{eg2}
Suppose that $N$ is a principal $H$-bundle over a compact base $M$, and suppose we are given the smooth action of a Lie group $G$ on $N$ commuting with the principal $H$-action.  Let $\psi$ be a $G$-invariant connection form on $N$ with curvature $\Psi = d\psi + \frac{1}{2}[\psi,\psi]$.  For any $Y\in\mathfrak{g}$, let $\Psi(Y) = \Psi - \psi(Y_N)$ denote the corresponding equivariant curvature.  Since $\Psi\in\mathcal{A}^2(N)$ is nilpotent, for any smooth function $\phi\in C^\infty(\mathfrak{h})$, the form $\phi(\Psi(Y))$ is given in terms of a Taylor series expansion about $\psi(Y_N)$.  Moreover, if $\phi$ is $H$-invariant, then $\phi(\Psi(Y))$ is basic, and thus defines a smooth $G$-equivariant differential form on $M$.

We obtain an $H\times G$-equivariant differential form with generalized coefficients $\delta(X-\Psi(Y))$ defined on $N$ by 
\begin{equation}\label{del}
\int_{\mathfrak{h}\times\mathfrak{g}}\delta(X-\Psi(Y))\phi(X,Y)dXdY = \vol(H,dX)\int_{\mathfrak{g}}\phi(\Psi(Y),Y)dY,
\end{equation}
for any compactly supported $\phi\in C^{\infty}(\mathfrak{h}
\times\mathfrak{g})$.  (The form is smooth with respect to the variable $Y\in\mathfrak{g}$.)
Using the fact that $\Psi$ is basic, we may define a corresponding form  $\delta_0(X-\Psi(Y))$ on $M$ by setting
\begin{equation}\label{del0}
\int_\mathfrak{h}\delta_0(X-\Psi(Y))\phi(X)dX = \vol(H,dX)\phi(\Psi(Y)),
\end{equation}
for any $H$-invariant $\phi\in C^{\infty}(\mathfrak{h})$.
\end{example}

\begin{example}\label{eg3}
The following example of an equivariant differential form with generalized 
coefficients is due to Paradan \cite{P1,P2,PV2,PV3}:

Let $N$ be a $G$-manifold, and let $\lambda$ be a smooth, $G$-invariant 
1-form on $N$.  Define the corresponding $\lambda$-moment map 
$f_\lambda:N\rightarrow \mathfrak{g}^*$ by
\begin{equation}
<f_\lambda,\xi> = -<\lambda, \xi_N>.\label{mommap}
\end{equation}
Then, on $N\setminus f_\lambda^{-1}(0)$, the form
\begin{equation}
\beta(\lambda)(X) = -i\lambda\int^\infty_0 e^{itD\lambda(X)} d t
\end{equation}
is well-defined as a $G$-equivariant form with generalized coefficients, 
and satisfies $D\beta_\lambda(X) = 1$ away from $f_\lambda^{-1}(0)$. (Roughly speaking, $\beta(\lambda)(X)$ is the generalized coefficients analog of the smooth form $\frac{\lambda}{D\lambda(X)}$ that appears in the proof of the Duistermaat-Heckman formula.)

Choose a $G$-equivariant neighbourhood $U$ of $f^{-1}_\lambda(0)$ in $N$, and let $\chi\in C^\infty(N)$ be a $G$-invariant cutoff function supported on $U$, such that $\chi\equiv 1$ on a smaller neighbourhood contained in $U$.  Define
\begin{equation}
P_{\lambda}(X) = \chi + d\chi\wedge\beta(\lambda).\label{Plam}
\end{equation}
Then $P_{\lambda}$ is a closed equivariant differential form with generalized coefficients supported in $U$, and its cohomology class in $\mathcal{H}^{-\infty}(\mathfrak{h},U)$ does not depend on the choice of cutoff function $\chi$ \cite{P1}.  We note that $P_{\lambda}(X) = 1 + D((\chi-1)\beta(\lambda)(X))$, so that $P_\lambda$ represents 1 in the generalized equivariant cohomology of $N$.
\end{example}

\begin{remark}\label{cwf}
Suppose that $N$ is a principal $H$-bundle equipped with an action of a Lie group $G$ commuting with the $H$-action, and a $G$-invariant connection 1-form $\psi$.  Define a 1-form $\nu$ on $N\times \mathfrak{h}^*$ by $\nu = <\xi,\psi>$, where $\xi$ denotes the $\mathfrak{h}^*$ variable.  We may then construct the $H\times G$-equivariant differential form with generalized coefficients $P_\nu(X,Y)$ using a sufficiently small neighbourhood $U$ of $N\times \{0\}$, and we have the following result from \cite{PV3}:
\begin{lemma}
Let $q:N\times\mathfrak{h}^*\rightarrow N$ denote projection onto the first factor.  Let $\psi_1,\ldots,\psi_r$ denote the components of $\psi$ with respect to some choice of basis for $\mathfrak{h}$.  If $\mathfrak{h}^*$ is oriented with respect to the corresponding dual basis, then
\begin{equation}\label{cweq}
q_*P_\nu(X,Y) = (2\pi i)^{\dim H} \delta(X-\Psi(Y))\frac{\psi_r\cdots\psi_1}{\vol(H,dX)}
\end{equation}
for any $(X,Y)\in \mathfrak{h}\times\mathfrak{g}$.
\end{lemma}
\end{remark}

\section{The differential form $\mathcal{J}(E,X)\in\mathcal{A}^{-\infty}(\mathfrak{h},M)$}\label{jsect}
Let $M$ be a smooth manifold, and let $E \subset TM$ denote a given distribution. We suppose that $E$ is co-oriented; that is, that its anihilator $E^0$ is oriented.  We suppose a Lie group $H$ acts on $M$ preserving $E$ and the orientation on $E^0$, such that the action is transverse to $E$.
Let $\imath:E^0\hookrightarrow T^*M$ denote inclusion, and let $q:E^0\rightarrow M$ denote projection.
\begin{definition}
We denote by $\mathcal{J}(E,X)\in\mathcal{A}^{-\infty}(\mathfrak{h},M)$ the equivariant differential form with generalized coefficients given by
\begin{equation}\label{J}
\mathcal{J}(E,X) = (2\pi i)^{-k}q_*\imath^*e^{iD\theta(X)},
\end{equation}
where $k=\rank E^0$.
\end{definition}
We note that $\imath^*D\theta(X) = d\theta_0 + f_{\theta_0}(X)$, where $\theta_0 = \imath^*\theta$, and $f_{\theta_0}:E^0\rightarrow \mathfrak{h}^*$ is defined as in \eqref{mommap}. By the theory of oscillatory integrals in \cite{Hor}, \eqref{J} defines an equivariant differential form with generalized coefficients on $M$, since $f_{\theta_0}^{-1}(0) = T^*_HM\cap E^0 = 0$.  (In the language of Berline and Vergne, the form $\exp(D\theta_0(X))$ is ``rapidly decreasing in $\mathfrak{h}$-mean'' along the fibres of $E^0$.)  We note also that the form $\mathcal{J}(E,X)$ is equivariantly closed: we have $D\mathcal{J}(E,X) = 0$.

We now wish to proceed with a local construction of the form defined in \eqref{J}.  Although the above definition suffices to obtain our results, the structure and properties of $\mathcal{J}(E,X)$ are revealed more clearly by this local description.

Let $U\subset M$ be a trivializing neighbourhood for $E^0$, and let $\balpha = (\alpha_1,\ldots,\alpha_k)\in \mathcal{A}^1(U)\otimes \R^k$ be a local oriented frame for $E^0|_U$.
Given such a choice of frame, we define a map
\begin{equation*}
f_{\balpha} :U\rightarrow \Hom(\mathfrak{h},\R^k)
\end{equation*}
by $f_{\balpha}(X) = -\balpha(X_M)$, for any $X\in \mathfrak{h}$, where $X_M$ is the fundamental vector field on $M$ generated by $X$. The equivariant differential of $\balpha$ is thus $D\balpha(X) = d\balpha + f_{\balpha}(X) \in\mathcal{A}^2(U)\otimes \R^k$.

Let $\delta_0 \in C^{-\infty}(\R^k)$ denote the Dirac delta function on $\R^k$.
Since $||\balpha||\neq 0$ on $U$, the transversality assumption ensures that $f_{\balpha}(m)$ is non-zero for all $m\in U$.  Thus, for any derivative $\delta^{(I)}_0$, the composition $\delta^{(I)}_0\circ f_{\balpha}(m)$ is well-defined as a generalized function on $\mathfrak{h}$ (see \cite{Hor, Mel}).

The expression $\delta_0(D\balpha(X))$ can be described in terms of its Taylor expansion as
\begin{align*}
\delta_0(D\balpha(X)) &= \delta_0(d\balpha + f_{\balpha}(X))\\
&=\sum^\infty_{|I|=0} \frac{\delta^{(I)}_0(f_{\balpha}(X))}{I!}d\alpha^I,
\end{align*}
where for $I = (i_1,\ldots,i_k)$, $I! = i_1!\cdots i_k!$, $|I| = i_1+\cdots +i_k$, $\delta_0^{(I)} = \left(\frac{\partial}{\partial x_1}\right)^{i_1}\cdots \left(\frac{\partial}{\partial x_k}\right)^{i_k}\delta_0$, and $d\alpha^I = d\alpha_1^{i_1}\wedge\cdots \wedge d\alpha_k^{i_k}$.

Since $\delta_0(D\balpha(X))$ is given in terms of the pullback of the Dirac delta function on $\R^k$, its pairing against a test function on $\mathfrak{g}$ depends on the map $f_{\balpha}$ and hence does not admit a simple description such as (\ref{del0}) above.  However, we can give the following representation of $\delta_0(D\balpha(X))$ in terms of the inverse Fourier transform:
\begin{equation}\label{fourier}
\delta_0(D\balpha(X)) = \frac{1}{(2\pi)^k}\int_{(\R^k)^*}e^{-i<\xi,D\balpha(X)>}d\xi,
\end{equation}
where $\left<\xi,D\balpha(X)\right> = \sum_{i=1}^k \xi^j(d\alpha_j+\alpha_j(X_M))$ and $d\xi = d\xi^1\cdots d\xi^k$ with respect to the basis for $(\R^k)^*$ dual to the one defined by the frame $\balpha$.

We now define an equivariant differential form with generalized coefficients on $U$ by
\begin{equation*}
\mathcal{J}_{\balpha}(E,X) = \alpha_k\wedge\cdots\wedge\alpha_1 \wedge\delta_0(D\balpha(X)).
\end{equation*}

\begin{lemma}\label{l1}
The form $\mathcal{J}_{\balpha}(E,X)$ does not depend on the choice of oriented frame $\balpha$.
\end{lemma}
\begin{proof}
Suppose that $\boldsymbol{\beta} = (\beta_1,\ldots ,\beta_k)$ is another frame for $E^0$ on $U$ defining the same orientation as $\balpha$.  Then we have $\boldsymbol{\beta} = A\balpha$ for some matrix $A$ with positive determinant, and so
\begin{align*}
\beta_k\wedge\cdots\wedge\beta_1\wedge\delta_0(D\boldsymbol{\beta}(X))& = \det(A)\alpha_k\wedge\cdots\wedge\alpha_1\wedge\delta_0(A(D\balpha(X)) + dA\wedge\balpha)\\
&=\alpha_k\wedge\cdots\wedge\alpha_1\wedge\delta_0(D\balpha(X) + A^{-1}(dA)\wedge\balpha)\\
&=\alpha_k\wedge\cdots\wedge\alpha_1\wedge\delta_0(D\balpha(X)),
\end{align*}
since $\det(A)\delta_0(A\mathbf{x}) = \delta_0(\mathbf{x})$ if $\det(A)>0$, and $\alpha_j\wedge\alpha_j = 0$ for all $j$.
\end{proof}
\begin{corollary}
There exists a well-defined equivariant differential form with generalized coefficients $\tilde{\mathcal{J}}(E,X)$ whose restriction to any trivializing neighbourhood $U\subset M$ is $\mathcal{J}_{\balpha}(E,X)$.
\end{corollary}
\begin{proposition}
The form $\tilde{\mathcal{J}}(E,X)$ is equivariantly closed.
\end{proposition}
\begin{proof}
In any local frame $\balpha$ we have
\begin{align*}
D\mathcal{J}_{\balpha}(E,X) & = (D(\alpha_k\cdots\alpha_1))\delta_0(D\balpha(X))\\
& = (\sum_{j=1}^k(-1)^{k-i}\alpha_k\cdots D\alpha_j(X)\cdots \alpha_1)\delta_0(D\balpha(X))\\
& = 0,
\end{align*}
thanks to the identity $u_j\delta_0(\mathbf{u}) = 0$ for $j=1\ldots k$.
\end{proof}
\begin{proposition}\label{pjnew}
We have the following equality of equivariant differential forms with generalized coefficients on $M$:
\begin{equation}
\tilde{\mathcal{J}}(E,X) = \mathcal{J}(E,X).\label{jtilde}
\end{equation}
\end{proposition}
\begin{proof}
We prove \eqref{jtilde} by showing that it holds on any choice of trivializing neighbourhood $U$.  Let $N=U\times (\R^k)^*$ denote the trivialization of the open subset $q^{-1}(0)$, and let $\xi$ denote the coordinate on $(\R^k)^*$. Define a 1-form $\lambda$ on $N$ by  $\lambda = -<\xi,\balpha>$.   It follows from the definition of the canonical 1-form on $T^*M$ that $\lambda$ coincides with $\imath^*\theta|_N$ under the identification $N\cong q^{-1}(U)$.
Let $t^1,\ldots, t^k$ be the basis for $\R^k$ with respect to which we have $\balpha = \sum\alpha_jt^j$.  If $\xi = \sum \xi^jt_j$ with respect to the corresponding dual basis for $(\R^k)^*$, then we have $\lambda = -\sum\xi^j\alpha_j$, and thus
\begin{equation*}
D\lambda(X) = \alpha_1\wedge d\xi^1 + \cdots + \alpha_k\wedge d\xi^k - <\xi, d\balpha + f_{\balpha}(X)>,
\end{equation*}
whence
\begin{equation*}
e^{iD\lambda(X)} = \alpha_k\cdots\alpha_1 e^{-i(\xi,D\balpha(X))}d\xi_1\cdots d\xi_k.
\end{equation*}
Thus, using \eqref{fourier}, we have
\begin{equation*}
\mathcal{J}(E,X)|_{U} = q_{*}e^{iD\theta_0(X)}|_U = \int_{(\R^k)^*}e^{iD\lambda(X)} = \mathcal{J}_{\balpha}(E,X) = \tilde{\mathcal{J}}(E,X)|_U.
\end{equation*}
\end{proof}
This completes our local description of the form $\mathcal{J}(E,X)$.  However, for our proof of the index formula in the next section, we chose to use the recent index theorem of Paradan and Vergne for transversally elliptic operators, for which we will need the following:
\begin{theorem}\label{l2}
Let $\theta$ denote the canonical 1-form on $T^*M$, and define $P_{\theta}(X)$ according to Example \ref{eg3} above.  Then we have the following equality of differential forms with generalized coefficients on $M$:
\begin{equation}
q_*\imath^*P_{\theta}(X) = (2\pi i)^k\mathcal{J}(E,X) = q_*\imath^*e^{iD\theta(X)}.\label{fibint}
\end{equation}
\end{theorem}
\begin{proof}
Let $\lambda$ denote the 1-form defined on $N=U\times (\R^k)^*$ as above, and define $P_\lambda(X)$ as in (\ref{Plam}), with $\chi$ an arbitrarily chosen cutoff funtion supported on a neighbourhood of $U\times\{0\}$ in $N$.
As in the proof of Proposition \ref{pjnew}, we have
\begin{equation*}
D\lambda(X) = \alpha_1\wedge d\xi^1 + \cdots + \alpha_k\wedge d\xi^k - <\xi, d\balpha + f_{\balpha}(X)>.
\end{equation*}
Let $\chi(\xi)$ be any arbitrary cutoff function supported on an open neighbourhood of $U\times \{0\}$ in $N$.
The contribution to the integral of $P_\lambda$ over $(\R^k)^*$ comes from the term of maximum degree in the $d\xi^j$.
We have
\begin{equation*}
P_\lambda(X) = \chi(\xi) - id\chi(\xi)\wedge\lambda \int^\infty_0 e^{itD\lambda(X)} dt,
\end{equation*}
where $\displaystyle d\chi(\xi) = \sum \frac{\partial \chi}{\partial \xi^i}(\xi) d\xi^i$, and
\begin{align*}
e^{it<\balpha, d\xi>} &=\prod_{j=1}^k (1+it \alpha_j\wedge d\xi^j)\\ 
&=(it)^k\alpha_1\wedge d\xi^1\wedge\cdots\wedge\alpha_k d\xi^k + (it)^{k-1}\sum_{i=1}^k \alpha_1\wedge d\xi^1\wedge\cdots\wedge\widehat{\alpha_j\wedge d\xi^j}\wedge\cdots\wedge \alpha_k \wedge d\xi^k\\
&+\quad\text{terms of lower degree.}
\end{align*}
We are thus interested in the top-degree part of $d\chi(\xi)\wedge\lambda \,e^{it<\balpha,d\xi>}$, which is given by
\begin{align*}
&\left(\sum_{i=1}^k \frac{\partial \chi}{\partial \xi^i}d\xi^i\right)\left(-\sum_{j=1}^k\xi^j\alpha_j\right)\left((it)^{k-1}\sum_{l=1}^k \alpha_1\wedge d\xi^1\wedge\cdots\wedge\widehat{\alpha_l\wedge d\xi^l}\wedge\cdots\wedge \alpha_k\wedge d\xi^k\right)\\
=&-(it)^{k-1}\left(\sum_{j=1}^k \frac{\partial \chi}{\partial \xi^j}\xi^j\right)\alpha_1\wedge d\xi^1\wedge\cdots\wedge\alpha^k \wedge d\xi^k.
\end{align*}
If we make the change of variables $\zeta^i = t\xi^i$, then the top-degree part of $-id\chi\wedge\lambda\int^\infty_0 e^{itD\lambda(X)}$ becomes
\begin{align*}
&i^k\int^\infty_0 \left(\frac{1}{t}\sum_{j=1}^k\frac{\partial \chi}{\partial \zeta^j}\left(\frac{\zeta}{t}\right)\zeta^j\right)\alpha_1\wedge d\zeta_1\cdots\alpha_k\wedge d\zeta^k e^{-i<\zeta, D\balpha(X)>}\\
=&i^k\int^\infty_0 \frac{d\phantom{t}}{dt}\left(-\chi\left(\frac{\zeta}{t}\right)\right)\alpha_k\cdots\alpha_1 e^{-i<\zeta, D\alpha(X)>}d\zeta^1\cdots d\zeta^k\\
=&i^k\alpha_k\cdots\alpha_1 e^{-<\zeta,D\balpha(X)>}d\zeta^1\cdots d\zeta^k.
\end{align*}
Integrating over $(\R^k)^*$ and using \eqref{fourier}, we obtain our result.
\end{proof}
\begin{remark}
We should remark that Theorem \ref{l2} is proved in a more general setting in \cite[Section 2.4]{P1}: Let $\mathcal{V}\rightarrow M$ be an $H$-equivariant vector bundle with compact base $M$, and denote by $\mathcal{H}^{-\infty}_c(\mathfrak{h},\mathcal{V})$ the space of equivariant differential forms on $\mathcal{V}$ with compact support.  Paradan shows that integration along the fibres of $\mathcal{V}$ defines a morphism $\mathcal{H}^{-\infty}_c(\mathfrak{h},\mathcal{V})\rightarrow \mathcal{H}^{\infty}(\mathfrak{h},M)$ that extends to equivariant differential forms that are rapidly decreasing in $\mathfrak{h}$-mean (such forms are defined in \cite{BV1}).

If $\lambda\in\mathcal{A}^1(\mathcal{V})$ is an invariant 1-form such that $f_\lambda^{-1}(0) = M$, then $P_\lambda(X)$ defines a class in $\mathcal{H}^{-\infty}_c(\mathfrak{h},\mathcal{V})$, while the form $e^{-iD\lambda(X)}$ is rapidly decreasing in $\mathfrak{h}$-mean.  Proposition 2.9 of \cite{P2} states that integration of either form over the fibres of $\mathcal{V}$ results in the same cohomology class in $\mathcal{H}^{-\infty}(\mathfrak{h},M)$.  While our result follows directly from this Proposition, we include the above proof since we are able to give an explicit local calculation on the level of differential forms.
\end{remark}

\begin{remark}
Let us consider the $H\times G$ equivariant form $\mathcal{J}(E,(X,Y))$ in the setting of Remark \ref{cwf}, 
where $E\subset TN$ is the space of horizontal vectors with respect to the 
connection 1-form $\theta$.
By (\ref{cweq}), the form $\mathcal{J}(E,(X,Y))$ is given by:
\begin{equation*}
\mathcal{J}(E,(X,Y)) = \delta(X-\Psi(Y))\frac{\psi_r\cdots\psi_1}{\vol(H,dX)},
\end{equation*}
and integrating over the fibres of $\pi:N\rightarrow M$, gives
\begin{equation}\label{cw2}
\pi_*\mathcal{J}(E,(X,Y)) = \delta_0(X-\Psi(Y)).
\end{equation}
Now, the coefficients of the above equivariant differntial form (\ref{cw2}) are generalized functions on $\mathfrak{h}$ supported at the origin, whence the pairing of this form against a smooth function of arbitrary support is well-defined.
Thus, given any invariant $f\in C^{\infty}(\mathfrak{h})$, we find using (\ref{cw2}) that
\begin{equation*}
\int_{\mathfrak{h}}\pi_*\mathcal{J}(E,(X,Y))f(X)dX = f(\Psi(Y)).
\end{equation*}
As a result, the equivariant Chern-Weil characteristic forms can be obtained by pairing invariant polynomials on $\mathfrak{h}$ against the fibre integral of $\mathcal{J}(E,(X,Y))$.
\end{remark}

\section{Index formulas}
\subsection{The general formula}
We wish to consider the following situation: suppose $E\subset TM$ is a sub-bundle of even rank, and that a compact Lie group $H$ acts on $M$ transverse to $E$.  We suppose given $H$-equivariant Hermitian vector bundles $\W^\pm\rightarrow M$, and an $H$-equivariant morphism $\sigma:\pi^*\W^+\rightarrow \pi^*\W^-$, where $\pi:T^*M\rightarrow M$ is the projection mapping.  We suppose that our symbol $\sigma$ ``depends only on $E^*$'' in the following sense: we have the exact sequence of vector bundles

\[
\xymatrix{0\ar[r] &E^0 \ar@{^{(}->}[r]^\imath &T^*M\ar@{->>}[r]^{r} &E^*\ar[r] &0,}
\]
and $\sigma$ is such that $\sigma = r^*\sigma_E$, for some symbol $\sigma_E:s^*\W^+\rightarrow s^*\W^-$, where $s:E^*\rightarrow M$ denotes projection onto $M$.

We note that if $\sigma_E$ is an elliptic symbol on $E^*$, then the support of $\sigma$ is $E^0$, and the transversality condition on the action of $H$ gives $E^0\cap T^*_HM = 0$, which means that $\sigma$ is $H$-transversally elliptic, in the sense of Atiyah \cite{AT}.
We may then compute its equivariant index using the formula of Paradan-Vergne \cite{PV3}:
\begin{equation}
\ind^H(\sigma)(he^Y) = \int_{T^*M(h)}(2\pi i)^{-\dim M(h)}\label{pv} \frac{\hat{A}^2(M(h),Y)}{D_h(\mathcal{N},Y)}\Ch(\sigma,h)(Y)P_{\theta^h}(Y).
\end{equation}
Let us recall the definitions of the terms appearing in the above formula. The 1-form $\theta^h$ is simply the restriction of the canonical 1-form on $T^*M$ to $T^*M(h)$.  The class $\Ch(\sigma,h)$ is the ``Chern character with support'' defined by Paradan and Vergne \cite{PV1,PV2,PV3}.  It is supported on $\supp(\sigma)$, and constructed using the Chern character
\begin{equation*}
\Ch_h(\mathbb{A},\sigma)(Y) = \Str(h^\W e^{\mathbb{F}_\sigma(Y)}),
\end{equation*}
where $\mathbb{A}$ is a superconnection on the $\mathbb{Z}_2$-graded bundle $\W = \W^+\oplus \W^-$ with no term in exterior degree zero, and $\mathbb{F}_\sigma(Y)$ is the equivariant curvature of the superconnection $\mathbb{A} + i\begin{pmatrix} 0 & \sigma \\ \sigma^* & 0\end{pmatrix}$.  We recall that the cohomology class of $\Ch_h(\mathbb{A},\sigma)$ depends only on the symbol $\sigma$ (see \cite{BV1}, Section 4.5), and note that this class coincides with that of $\Ch(\sigma,h)$ in an appropriate cohomology space \cite{PV1,PV2}.

Let $\nabla$ be an $H$-equivariant connection on $TM$. For any $h\in H$, we have the decomposition
\begin{equation*}
TM|_{M(h)} = TM(h)\oplus\mathcal{N}.
\end{equation*}
The connection $\nabla$ induces connections $\nabla^0$ and $\nabla^1$ on $TM(h)$ and $\mathcal{N}$, respectively, with corresponding equivariant curvatures $R_0(Y)$ and $R_1(Y)$, for $Y\in\mathfrak{h}(h)$.
We obtain the smooth, closed $H(h)$-equivariant forms on $M(h)$ defined by
\begin{equation*}
\hat{A}^2(M(h),Y) = \det\left(\frac{R_0(Y)}{e^{R_0(Y)/2}-e^{-R_0(Y)/2}}\right),
\end{equation*}
for $Y\in\mathfrak{h}(h)$ sufficiently small, and
\begin{equation*}
D_h(\mathcal{N},Y) = \det\left(1-h^{\mathcal{N}}e^{R_1(Y)}\right),
\end{equation*}
where $h^{\mathcal{N}}$ denotes the linear action induced by $h$ on $\mathcal{N}$.

\begin{remark}
We are allowing an abuse of notation in our statement of the formula (\ref{pv}) above.  The $H$-equivariant index of $\sigma$ is an $H$-invariant generalized function on $H$.  The right hand side of the above formula in fact defines an $H$-invariant generalized function $\ind^H(\sigma)_h(Y)$ on the tubular neighbourhood $H\times_{H(h)}\mathcal{U}_h\hookrightarrow H$, where $\mathcal{U}_h$ is an open $H(h)$-invariant neighbourhood of 0 in $\mathfrak{h}(h)$.  For smooth functions the equality $f(he^Y) = f_h(Y)$ follows from the localization formula in equivariant cohomology, but in the case of generalized functions one must carefully check compatibility conditions using the descent method of Duflo-Vergne \cite{DV}.  When the index formula of Berline and Vergne is used, this checking was done in \cite{BV1,BV2}.  Paradan and Vergne solve this problem in \cite{PV3} by proving a localization formula in equivariant cohomology with generalized coefficients that allows the right-hand sides of (\ref{pv}) to be patched together to give a generalized function on $H$.
\end{remark}

Our goal is to compute the pushforward of the formula (\ref{pv}) on $T^*M$ to obtain a formula as an integral over $M$.  Using a transgression argument similar to that in \cite[Proposition 3.38]{PV2} (see also \cite[Proposition 3.11]{P1}, \cite[Proposition 2.6]{P2}), we have:

\begin{proposition}\label{propa}
If the 1-forms $\alpha, \beta\in \mathcal{A}(T^*M)$ agree on $\supp(\sigma)$, then the following equality holds in $\mathcal{H}^{-\infty}(\mathfrak{h},T^*M)$:
\begin{equation}
\Ch(\sigma)(Y)P_{\alpha}(Y) = \Ch(\sigma)(Y)P_{\beta}(Y).
\end{equation}
\end{proposition}

Now, let $\W = \W^+\oplus \W^-$ be a $\mathbb{Z}_2$-graded $H$-equivariant Hermitian vector bundle, and suppose $\sigma_E: s^*\W^+\rightarrow s^*\W^-$ is an $H$-invariant elliptic symbol on $E^*\subset T^*M$.   We let $s$ continue to denote the restriction of $s$ to $T^*M(h)$.  If we let $\sigma = r^*\sigma_E$, then we have:
\begin{theorem}
If the Lie group $H$ acts on $M$ transverse to the distribution $E\subset TM$, then the symbol $\sigma = r^*\sigma_E$ is transversally elliptic, and
the $H$-equivariant index of $\sigma$ is the generalized function on $H$ whose germ at $h\in H$ is given, for $Y\in\mathfrak{h}(h)$ sufficiently small, by
\begin{equation}
\ind^H(\sigma)(he^Y) = \int_{M(h)}(2\pi i)^{-\rank E(h)}\frac{\hat{A}^2(M(h),Y)}{D_h(\mathcal{N},Y)}\mathcal{J}(E(h),Y)s_*\Ch(\sigma_E,h)(Y).\label{main}
\end{equation}
\end{theorem}
\begin{proof}
Since the symbol $\sigma$ is the pullback to $T^*M$ of the elliptic symbol $\sigma_E$ on $E^*$, we have $\Ch(\sigma,h) = r^*(\Ch(\sigma_1,h))$. Denote by $\theta^h$ the restriction of the canonical 1-form $\theta$ to $T^*M(h)$.
The restriction of $\sigma_E$ to $E^*(h)$ is again elliptic, and $\sigma|_{T^*M(h)}$ has support $E^0(h)$. 
By Proposition \ref{fix}, we know that the action of $H(h)$ on $M(h)$ is  
transverse to $E(h)\subset TM(h)$, whence the restriction of $\sigma$ to $T^*M(h)$ is transversally elliptic.  Moreover, since the action of $H(h)$ is transverse to $E(h)$, the form $\mathcal{J}(E(h),Y)$ is well-defined as an $H(h)$-equivariant differential form with generalized coefficients on $M(h)$.
We choose an $H$-equivariant splitting $T^*M = E^*\oplus E^0$, giving us the commutative diagram
\[
\xymatrix{ &T^*M(h) \ar[ld]_{p} \ar[dd]^{\pi} \ar[rd]^{r} &\\
	E^0(h) \ar[rd]_{q} & &E^*(h) \ar[ld]^{s}\\
	&M(h)&}
\] 
Since $\theta^h$ and $\imath^*\theta^h$ agree on $E^0(h)$, we may use Proposition \ref{propa} to obtain
\begin{equation*}
\Ch(\sigma,h)(Y)P_{\theta^h}(Y) = r^*\Ch(\sigma_E,h)(Y)p^*\imath^*P_{\theta^h}(Y)
\end{equation*}
Using Theorem \ref{l2} and the commutative diagram, we see that
\begin{equation*}
r_*p^*\imath^*P_{\theta^h}(Y) = s^*q_*\imath^*P_{\theta^h}(Y)) = (2\pi i)^{\rank E^0(h)}s^*\mathcal{J}(E(h),Y),
\end{equation*}
and thus,
\begin{align*}
\pi_*\Ch(\sigma, h)(Y)P_{\theta^h}(Y) &= s_*r_*\left(r^*\Ch(\sigma_E,h)(Y)p^*\imath^*P_{\theta^h}(Y)\right)\\
&=s_*\left(\Ch(\sigma_E,h)(Y)r_*p^*\imath^*P_{\theta^h}(Y)\right)\\
&=(2\pi i)^{\rank E^0} s_*\left(\Ch(\sigma_E,h)(Y)s^*\mathcal{J}(E(h),Y)\right).
\end{align*}
By integrating over the fibres in \eqref{pv} and substituting the above, the result follows.
\end{proof}

\begin{remark}
Suppose $\mathcal{V}$ is some $H$-equivariant vector bundle on $M$, and consider the symbol $\sigma_{\mathcal{V}} = \sigma\otimes\Id_\mathcal{V}:\pi^*\W^+\otimes \mathcal{V}\rightarrow \pi^*\mathcal{W}^-\otimes \mathcal{V}$.  Using the multiplicativity of the Chern character \cite{PV0,PV1}, we have $\Ch(\sigma_{\mathcal{V}}) = \Ch(\sigma)\Ch(\mathcal{V})$.  Since $\Ch(\mathcal{V})$ is a form on $M$, we obtain the following extension to \ref{main}:
\begin{proposition}The $H$-equivariant index of $\sigma_{\mathcal{V}}$ is given, for $X\in\mathfrak{h}$ sufficiently small, by\label{sigF}
\begin{equation}
\ind^H(\sigma_{\mathcal{V}})(e^X) = \frac{1}{(2\pi i)^{\rank E}}\int_M \hat{A}^2(M,X)\mathcal{J}(E,X)\Ch(\mathcal{V})(X)s_*\Ch(\sigma_E)(X),
\end{equation}
with similar formulas near other elements $h\in H$.
\end{proposition}
\end{remark}

\subsection{Almost-CR structures}\label{comp}
We now wish to consider the special case where the sub-bundle $E$ carries a complex structure.  In these cases if we define our bundle $\W$ in terms of the bundle of differential forms on $M$, and make specific choices for the connection on $TM$ and superconnection on $\W$, the fibre integral of the Chern character can be computed explicitly, giving us a formula involving only $\mathcal{J}(E,X)$ and equivariant characteristic classes on $M$.

Let $\rank E = 2m$ (so that $\dim(M) = 2m+k$), and suppose that $M$ is equipped with an $H$-invariant almost-CR structure $E\otimes\C = E_{1,0}\oplus E_{0,1}$. Using the same approach as in \cite{F}, we let $E^*\otimes\C = E^{1,0}\oplus E^{0,1}$ denote the decomposition of $E^*\otimes\C$ induced by the almost-CR structure, and take $\W^\pm = \bigwedge^{even/odd} E^{0,1}$ as our Hermitian vector bundles.  The bundle $\W$ becomes a spinor module for the Clifford multiplication  $\mathbf{c}: E\rightarrow \End_\C(\W)$,  and we may take our symbol 
\begin{equation*}
\sigma:\pi^*\W^+\rightarrow \pi^*\W^-
\end{equation*}
to be $\sigma(x,\xi) = \mathbf{c}(r(\xi))_x$, which has support equal to $E^0$.

We equip $E_{1,0}$ with an $H$-invariant Hermitian metric, and $H$-invariant Hermitian connection $\nabla$.  Letting $\nabla^\W$ denote the connection on $\W$ induced by $\nabla$, we consider Quillen's superconnection \cite{Q}
\begin{equation}
\mathbb{A}(\sigma) = \pi^*\nabla^\W +iv_\sigma
\end{equation}
on $\pi^*\W$, where $v_\sigma = \begin{pmatrix} 0 & \sigma^*\\ \sigma & 0\end{pmatrix}$.  Here, $\sigma^*$ is defined using the Hermitian metric, so that $v_\sigma^2(x,\xi) = -||p_1(\xi)||_x^2 \Id$.

Let $\mathbb{F}(\mathbb{A}(\sigma))(X)$ denote the equivariant curvature of $\mathbb{A}(\sigma)$, and define the equivariant Chern character
\begin{equation*}
\Ch(\mathbb{A}(\sigma)) = \Str(e^{\mathbb{F}(\mathbb{A}(\sigma))}),
\end{equation*}
where $\Str$ denotes the supertrace \cite{BGV,Q}.
By \cite[Proposition 6.13]{PV1}, for any $X\in \mathfrak{h}$, we have
\begin{equation}
\Ch(\mathbb{A}(\sigma))(X) = (2\pi i)^{2m}\Td(E^{1,0},X)^{-1}\Th_{MQ}(E^*,X),\label{thom}
\end{equation}
where $\Th_{MQ}(E^*,X)$ is an equivariant Thom form on $E^*$ defined similarly to the ``gaussian-shaped'' Thom form in \cite{MQ}.

Similarly, let $j:T^*M(h)\rightarrow T^*M$ denote inclusion of the $h$-fixed points.  The Chern character on $T^*M(h)$ is given as in \cite{BV1} by
\begin{equation*}
\Ch_h(\mathbb{A}(\sigma)) = \Str(h\cdot j^*e^{\mathbb{F}(\mathbb{A}(\sigma))}).
\end{equation*}

The class $\Ch_{sup}(\sigma, h)$ can be represented in calculations by the form
\begin{equation}
c(\mathbb{A},\sigma,\chi) = \chi\Ch_h(\mathbb{A},\sigma) + d\chi \beta(\mathbb{A},\sigma),
\end{equation}
where $\chi$ is a smooth cutoff function supported in a sufficiently small neighbourhood of $\supp(\sigma)$ \cite{PV1,PV3}.  Moreover, the integrals of $\Ch_h(\mathbb{A}(\sigma))$ and $\Ch(\sigma, h)$ over the fibres of $E^*(h)$ coincide in $\mathcal{H}^\infty(\mathfrak{h}(h),M(h))$.

We now consider the restrictions of the subbundles $E^*$ and $E^0$ to $M(h)$.  Let $\mathcal{N}^*_0 = \mathcal{N}^*\cap(E^0|_{M(h)})$ and let $\mathcal{N}^*_E = \mathcal{N}^*\cap(E^*|_{M(h)})$. Let $\nabla = \nabla_h\oplus \nabla_1$ denote the decomposition of $\nabla$ on $E^*|_{M(h)} = E^*(h)\oplus \mathcal{N}^*_E$.  If $\nabla_h^{\W}$ denotes the connection on $\W(h)$ induced by $\nabla_h$, then when pulled back to $T^*M(h)$, the superconnection $\mathbb{A}(\sigma)$ decomposes according to $j^*\mathbb{A}(\sigma) = \mathbb{A}_h(\sigma)\oplus \pi^*\nabla^\W_1$, where 
$\mathbb{A}_h(\sigma) = \pi^*\nabla_h^{\W} + iv_\sigma|_{T^*M(h)}$.

Since the action of $h$ on $\W(h)$ is trivial, we see that
\begin{equation*}
\Ch_h(\mathbb{A}(\sigma))(X) = \Ch(\mathbb{A}_h(\sigma))(X)\Str(h\cdot e^{\pi^*F(\nabla_1)(X)}),
\end{equation*}
where $F(\nabla_1)(X)$ denotes the equivariant curvature of $\nabla_1$.

It was shown in \cite{F} that 
\begin{equation*}
\Str(h\cdot e^{\pi^* F(\nabla_1)(X)}) = \det\nolimits_\C(1-h\cdot e^{-\pi^*F(\nabla_1)(X)}).
\end{equation*}

The bundle $\mathcal{N}_E$ carries a complex structure induced from that on $E$.  Let $D_h^\C(\mathcal{N}_E)$ denote the form on $M$ defined as above but using the complex determinant in place of the real determinant.  Let $\overline{\mathcal{N}}$ denote the complex conjugate of $\mathcal{N}$, and note that the Hermitian metric allows the identification $\overline{\mathcal{N}} = \mathcal{N}^*$.  Choosing a splitting $\mathcal{N} = \mathcal{N}_0\oplus \mathcal{N}_E$, we can write 
\begin{equation*}
D_h(\mathcal{N}) = D_h(\mathcal{N}_0)D_h^\C(\mathcal{N}_E)D_h^\C(\overline{\mathcal{N}_E}),
\end{equation*}
while 
\begin{equation*}
\det\nolimits_\C(1-h\cdot e^{-\pi^*F(\nabla_1)(X)}) = \pi^*D_h^\C(\mathcal{N}_E^*)=\pi^*D_h^\C(\overline{\mathcal{N}_E}).
\end{equation*}
Finally, we apply the decomposition (\ref{thom}) to the Chern character $\Ch(\mathbb{A}_h(\sigma))$ on $E^*(h)$, noting as well that
\begin{equation*}
\hat{A}^2(M(h)) = \hat{A}^2(E(h))\hat{A}^2(TM(h)/E(h)) = \Td(E(h))\Td(E^*(h))\hat{A}^2(E^0(h)).
\end{equation*}
We substitute all of the above into the formula (\ref{main}), and note that $s_*\Th_{MQ}(E^*,X) = 1$ to obtain

\begin{theorem}\label{rrf}
Suppose that a compact Lie group $H$ acts on $M$ transverse to $E\subset TM$, and suppose that $M$ is equipped with an $H$-invariant almost-CR structure $E\otimes\C = E_{1,0}\oplus E_{0,1}$.  Then we have
\begin{equation}
\ind^H(\sigma)(h e^Y)  =
\int_{M(h)} (2\pi i)^{-\rank E(h)/2}\frac{\Td(E(h),Y)}{D_h^\C(\mathcal{N}_E,Y)}\frac{\hat{A}^2(E^0(h),Y)}{D_h(\mathcal{N}_0,Y)}\mathcal{J}(E(h),Y). 
\end{equation}
\end{theorem}
\section{Examples}
\subsection{Contact manifolds}
In \cite{F}, the case of a compact, co-oriented contact manifold $(M,E)$, upon which a Lie group $G$ acts transverse to $E$, was considered.  Here, the anihilator $E^0$ is trivial, and a global non-vanishing contact form $\alpha\in\Gamma(E^0)$ may be chosen.

Since $d\alpha|_E$ is symplectic, we can choose a compatible $G$-invariant complex structure $J$ on $E$, putting us in the setting of Section \ref{comp}, with the added simplification of the trivialization $E^0 = M\times \R$ defined by the choice of contact form.

The index formula of Theorem \ref{rrf} is then
\begin{equation}
\ind^G(\sigma)(he^Y) = \int_{M(h)}(2\pi i)^{-k(h)}\frac{\Td(E(h),Y)}{D^{\C}_h(\mathcal{N},Y)}\mathcal{J}(E(h),Y),
\end{equation}
where $2k(h)+1 = \dim M(h)$.

In the contact case, we may write $\mathcal{J}(E(h),Y) = \alpha_h\delta_0(D\alpha_h(Y))$, where $\alpha_h = \alpha|_{M(h)}$.  The properties of $\mathcal{J}(E(h),Y)$ ensure that it is independent of the choice of contact form $\alpha$.

\begin{remark}
We note that in the special case of a Sasakian manifold, the Reeb vector field associated to our choice of contact form is Killing.  Thus, the group of isometries automatically provides us with an action transverse to the contact distribution.
\end{remark}

\subsection{Almost contact manifolds}
The proof given in \cite{F} relies on the additional structure one has on a contact manifold.  In particular, it makes use of both the symplectic structure on the contact distrubution, and the trvialization of the anihilator line bundle.
Since our proof dispenses with these assumptions, we can apply our index formula in a number of more general settings, such as that of an almost-contact manifold.

We say $(M,\varphi,\xi,\eta)$ is an almost contact manifold if  $\eta\in \mathcal{A}^1(M)$ and $\xi\in \Gamma(TM)$ satisfy $\eta(\xi)=1$, and $\varphi\in \Gamma(\End(TM))$ satisfies $\varphi^2 = -\Id + \eta\otimes\xi$.

It is shown in \cite{Blair} that one can always find a compatible metric $g$ such that $g(\varphi X,\varphi Y) = g(X,Y) - \eta(X)\eta(Y)$.  If a Lie group $G$ acts transverse to $E = \ker(\eta)$ and preserving the tensors $\varphi$, $\xi$ and $\eta$, the results of Section \ref{comp} apply, since the $\pm i$-eigenbundles of $\varphi|_E$ define an almost-CR structure.

Let $\nabla$ denote the Levi-Civita connection with respect to $g$. If $(\nabla_X\phi)X = 0$, $M$ is called a {\em nearly cosymplectic} manifold, and it follows that $\xi$ is Killing \cite{Blair}, and thus that the group of isometries of $(M,g)$ acts transverse to $E$.
\begin{remark}
If one defines a two-form $\Phi$ by $\Phi(X,Y) = g(X,\varphi Y)$, then $\eta\wedge\Phi^n$ is a volume form on $M$ (this is an alternative definition of an almost-contact structure).  One could extend this notion to distributions of higher corank, by supposing the existence of 1-forms $\eta_1,\ldots, \eta_k$ and a 2-form $\Phi$ on a manifold of dimension $2n+k$, such that $\eta_1\wedge\cdots\wedge\eta_k\wedge\Phi^n\neq 0$; the formula of Theorem \ref{rrf} would still apply, provided that the forms $\eta_i$ were $H$-invariant, and not contained in $T^*_HM$.
\end{remark}
\subsection{(Locally) free actions}\label{vergnecomp}
As mentioned in Remark \ref{lfa} above, the case where $\rank E^0 = \dim H$ corresponds to the simpler case of a locally free action.  Here, the natural choice for our distribution $E$ is the space of horizontal vectors with respect to some choice of connection on $M$, and the anihilator $E^0$ becomes the trivial bundle $E^0 = M\times\mathfrak{h}^*$.  Using the notation of Remark \ref{cwf}, we have the global expression
\begin{equation*}
\mathcal{J}(E,X) = \psi_r\cdots\psi_1\delta_0(D\psi(X)) = \psi_r\cdots\psi_1\delta_0(X-\Psi).
\end{equation*}
Let $\pi:M\rightarrow B=M/G$ denote the quotient mapping.  Let $j_{\mathfrak{h}}(X)$ denote the function
\begin{equation*}
j_\mathfrak{h}(X) = \det\nolimits_{\mathfrak{h}}\frac{e^{\ad X/2} - e^{-\ad X/2}}{\ad X}.
\end{equation*}
We have the Schur orthogonality formula \cite{PV3} 
\begin{equation*}
\delta_0(X-\Psi) = j_\mathfrak{h}(X)\sum_{\tau\in\hat{H}}\Tr \tau(e^X) \Tr \tau^*(e^\Psi),
\end{equation*}
and the identity $\hat{A}^2(M,X) = j_\mathfrak{h}(X)\pi^*\hat{A}^2(B)$.  Since $\sigma_E$ is defined on $E^* = \pi^*T^*B$, $s_*\Ch(\sigma_E)$ is the pull-back of a form on $B$.
From \cite{V}, we have the formula
\begin{equation*}
\int_{M/H}\alpha = \int_M |S|\alpha\wedge v_\mathfrak{h},
\end{equation*}
where $|S|$ is the order of the generic stabilizer, and $v_\mathfrak{h} = \vol(H)\psi_r\cdots\psi_1$.

Combining the above, when the $H$-action is locally free, we obtain the expansion
\begin{align*}
\ind^H(\sigma)(e^X) &= \frac{1}{(2\pi i)^{\rank E}}\int_M \hat{A}^2(M,X)s_*\Ch(\sigma_E)(X)\mathcal{J}(E,X)\\
&= \sum_{\tau\in\hat{H}}\Tr \tau(e^X)\frac{1}{(2\pi i)^{\dim M/H}}\int_{M/H}\hat{A}^2(M/H)|S|^{-1}s_*\Ch(\sigma_E)\Tr \tau^*(e^\Psi),
\end{align*}
with similar formulas near other elements of $H$.  By \cite{V}, the expression 
\begin{equation*}
\frac{1}{(2\pi i)^{\dim M/H}}\int_{M/H}\hat{A}^2(M/H)|S|^{-1}s_*\Ch(\sigma_E)
\end{equation*}
defines the index of $\sigma_E$ on the orbifold $M/H$.

\subsection{Induced representations}
We now consider the following setting: Suppose $G$ is a compact semi-simple Lie group, and $H$ is a closed subgroup of $G$.  We let $M=G/H$, on which $G$ acts transitively.  Suppose $\tau:H\rightarrow \End(V)$ is a finite-dimensional, irreducible unitary representation of $H$.  Denote by $\mathcal{V}_\tau = G\times_\tau V$ the corresponding vector bundle over $M$.
One may then define the induced representation $\inl2_H^G(\tau)$
of $G$ on the $L^2$-sections of $\mathcal{V}_\tau\rightarrow M$ \cite{Kn}.
The character of this representation is a generalized function on $G$.  Berline and Vergne \cite{BV0} gave a formula for this character as an equivariant index, as follows:

Since $G$ acts transitively on $M$, every differential operator is transversally elliptic, including the zero operator
\begin{equation*}
0_\tau:\Gamma_{L^2}(\mathcal{V}_\tau)\rightarrow 0.
\end{equation*}
Its index, in terms of the Berline-Vergne index formula \cite{BV1,BV2}, is given near $g\in G$ by
\begin{equation}\label{bvind}
\ind^G(0_\tau)(ge^X) = \int_{T^*M(g)}(2\pi i)^{-\dim M(g)}\frac{\hat{A}^2(M(g),X)}{D_g(\mathcal{N},X)}\Ch_g(\mathcal{V}_\tau,X)e^{iD\theta^g(X)},
\end{equation}
where $\theta$ is the canonical 1-form on $T^*M$, and $\theta^g$ denotes its restriction to $T^*M(g)$.

The main result of \cite{BV0} is the identity
\begin{equation}\label{bvchar}
\chi(\inl2^G_H(\tau))(g) = \ind^G(0_\tau)(g).
\end{equation}
Now suppose that $M=G/H$ is Hermitian; that is, we suppose that $M$ is a complex manifold.  We may equivalently write $M = G^\C/P$, where $G^\C$ denotes the complexification of $G$, and $P$ is a parabolic subgroup \cite{Kn}.  (We may, for example, take $H$ to be a maximal torus $T$.)

As shown by Bott \cite{Bott}, if $\tau:H\rightarrow \End(V)$ is a holomorphic representation of $H$ on a finite-dimensional complex vector space $V$, then $\mathcal{V}_\tau = G\times_\tau V$ is a holomorphic vector bundle over $M=G/H$.  In this setting we may define the ``holomorphic induced representation'' of $G$ on the holomorphic sections of $\mathcal{V}_\tau$, which we'll denote by $\hol^G_H(\tau)$, following \cite{Kn}.

The $G$-action on $\mathcal{V_\tau}$ induces a $G$-module structure on the cohomology spaces $H^q(M, \mathcal{O}(\mathcal{V}_\tau))$, where $\mathcal{O}(\mathcal{V}_\tau)$ denotes the sheaf of holomorphic sections of $\mathcal{V}_\tau$.
Bott showed that if $\tau$ is irreducible, then the above cohomolgy spaces vanish in all but one degree, 
and that the non-vanishing space $H^p(M,\mathcal{O}(\mathcal{V}_\tau))$ is an irreducible $G$-representation.
The character of this representation can be computed using the equivariant Riemann-Roch theorem on the complex manifold $M$.  If $\sigma$ denotes the symbol of the Dolbeault-Dirac operator on sections of $\W = \bigwedge T^{0,1}M$ then the character of $\hol^G_H(\tau)$ is the index of the symbol $\sigma_\tau = \sigma\otimes \Id$ on $\pi^*(\W\otimes\mathcal{V}_\tau)$:
\begin{equation*}
\chi(\hol^G_H(\tau))(ge^X) = \ind^G(\sigma_\tau)(ge^X) = \int_{M(g)} (2\pi i)^{-\dim M(g)/2}\frac{\Td(TM(g),X)}{D_g^\C(\mathcal{N},X)}\Ch_g(\mathcal{V}_\tau,X).
\end{equation*}

The two cases given above represent two extremes of transversally elliptic symbols: the zero symbol in the first case, whose support is all of $T^*M$, and an elliptic symbol in the second, whose support is the zero section.  We may also consider the following intermediate possibility: we suppose there exists a $G$-invariant sub-bundle $E\subset TM$, and a symbol whose support is $E^0\subset T^*M$.

If $G/H$ is Hermitian, then $TM = G\times_H\mathfrak{h}^\bot$, and $\mathfrak{h}^\bot$ is a complex vector space.  We choose some complex, $H$-invariant subspace $W$ of $\mathfrak{h}^\bot$, and let $E=G\times_HW\subset TM$.  For example, if $H$ is a maximal torus, then we may take $W$ to be a sum of root spaces.

We are now in the setting of Section \ref{comp} above: $E\subset TM$ is $G$-invariant and equipped with a complex structure.  Since $G$ acts transitively on $M$, the action is automatically transverse to $E$. We let $\W = \bigwedge E^{0,1}$. Suppose $\tau:H\rightarrow \End(V)$ is a finite-dimensional unitary irreducible $H$-representation, and let $\mathcal{V}_\tau = G\times_\tau V$.  If we consider the symbol $\sigma_\tau$ on $\pi^*(\W\otimes\mathcal{V}_\tau)$, then Propositions \ref{sigF} and \ref{rrf} give
\begin{equation}
\ind^G(\sigma_\tau)(ge^X) = \int_{M(g)} (2\pi i)^{-\rank E/2} \frac{\Td(E(g),X)}{D_g^\C(\mathcal{N}_E,X)}\frac{\hat{A}^2(E^0(g),X)}{D_g(\mathcal{N}_0,X)}\mathcal{J}(E(g),X)\Ch_g(\mathcal{V}_\tau,X).\label{indrep}
\end{equation}
As special cases, we have
\begin{enumerate}
\item $E=0$: This is the case of the zero operator on sections of $\mathcal{V}_\tau$.  We have $E^0 = T^*M$, $\mathcal{N}_0 = \mathcal{N}$, and $\mathcal{N}_E = \{0\}$.  If we let $\displaystyle \mathcal{J}(M,X) = (2\pi i)^{-\dim M}\pi_*e^{iD\theta(X)}$ denote the form corresponding to the zero section, then (\ref{indrep}) becomes
\begin{align*}
\ind^G(0)(ge^X) & =\int_{M(g)} \frac{\hat{A}^2(M(g),X)}{D_g(\mathcal{N},X)}\Ch_g(\mathcal{V}_\tau,X)\mathcal{J}(M(g),X)\\
& =\int_{T^*M(g)}(2\pi i)^{-\dim M(g)}\frac{\hat{A}^2(M(g),X)}{D_g(\mathcal{N},X)}\Ch_g(\mathcal{V}_\tau,X)e^{iD\theta^g(X)}\\
& = \chi(\inl2^G_H(\tau))(ge^X),
\end{align*}
by the Berline-Vergne character formula (\ref{bvchar}).
\item $E=TM$: This is the case of the Dolbeault-Dirac operator on sections of $\bigwedge T^{0,1}M$, twisted by the bundle $\mathcal{V}_\tau$.  In this case we have $E^0=0$, and so $\mathcal{N}_0 = \{0\}$, $\mathcal{N}_E = \mathcal{N}$, and $\mathcal{J}(E,X)=1$, and thus (\ref{indrep}) becomes
\begin{equation*}
\ind^G(\sigma_\tau)(ge^X) = \int_{M(g)}(2\pi i)^{-\dim M(g)/2} \frac{\Td(TM(g),X)}{D_h^\C(\mathcal{N},X)}\Ch_g(\mathcal{V}_\tau,X),
\end{equation*}
and we recover the Riemann-Roch formula for the character of the holomorphic induced representation.
\end{enumerate}
\section*{Acknowledgements}
The author would like to thank Eckhard Meinrenken for many fruitful 
discussions during which many of the ideas devloped above were first 
suggested, as well as much helpful feedback during the preparation of 
this article.  The local form $\mathcal{J}_{\balpha}(E,X)$ as the 
natural extension of the corresponding form appearing in \cite{F}, and its
relation to Chern-Weil forms were both observations of Meinrenken.

\bibliographystyle{plain}
\bibliography{index}

\end{document}